\documentclass[11pt]{amsart}
\usepackage[numbers]{natbib}
\usepackage{url}
\usepackage{graphicx}
\usepackage{amsfonts, amssymb, amsmath, amsthm}
\usepackage{multicol}
\usepackage{tikz}
\usepackage{tikz-cd}
\usepackage{tikz-3dplot}
\usepackage{cancel}
\usepackage{changepage}
\usepackage{csquotes}
\usepackage{adjustbox}
\usetikzlibrary{matrix}

\newtheorem*{mtheorem}{Main Theorem}
\newtheorem*{theorem}{Theorem}
\newtheorem*{lemma}{Lemma}
\newtheorem*{prop}{Proposition}

\newtheorem*{claim}{Claim}

\theoremstyle{definition}
\newtheorem*{definition}{Definition}

\newtheorem*{rmk}{Remark}

\newcommand{\ang}{\measuredangle}

\newcommand{\FF}{\mathcal{F}}
\newcommand{\bdy}{\partial}

\newcommand{\no}{\noindent}
\newcommand{\wtilde}{\widetilde}

\begin{document}
\title{Point Leaf Maximal Singular Riemannian Foliations in Positive Curvature}

\author{Adam Moreno}
\address{Department of Mathematics\\
University of Notre Dame\\
      Notre Dame, IN 46556\\USA
      }
\email{amoreno3@nd.edu}

\maketitle

\begin{abstract}
We generalize the notion of fixed point homogeneous isometric group actions to the context of singular Riemannian foliations. We find that in some cases, positively curved manifolds admitting these so-called \textit{point leaf maximal} SRF's are diffeo/homeomorphic to compact rank one symmetric spaces. In all cases, manifolds admitting such foliations are cohomology CROSSes or finite quotients of them. Among non-simply connected manifolds, we find examples of such foliations which are non-homogeneous. 
\end{abstract}

\section{Introduction}

\indent The link between symmetry and positive curvature has been studied heavily since the Grove Symmetry Program began in 1991. The hope is to find, among other things, either a topological obstruction to positive curvature, or new examples of manifolds which admit positive curvature metrics. The basic start-up kit is this: (1) Pick a notion of ``high symmetry,'' then (2) say as much as you can topologically about positively curved manifolds with the chosen type of symmetry. Perhaps the most appealing feature is that the user is free to choose the notion of symmetry (e.g. low cohomogeneity, high isometry rank, etc.) and "classification" then means the best you can do (e.g. homeo/diffeo-morphism, tangential homotopy equivalence, etc.). This systematic approach has yielded several results, for which the reader in encouraged to see \cite{grove2008developments} and \cite{wilking2006positively}.

One way to measure symmetry is by the \textit{cohomogeneity} of an action, or the dimension of the resulting orbit space. At one extreme are the cohomogeneity 0 manifolds - connected manifolds on which a group $G$ acts transitively (i.e. homogeneous manifolds). Positively curved such manifolds were classified long before the symmetry program through the combined efforts of Aloff, Berard-Bergary, Berger, and Wallach in the 70's (see \cite{wilking2015revisiting}). It seems natural enough to then ask, ``how large can the orbits of an isometric action be in the presence of singular orbits, and what does this tell us about our manifold?'' For example, when the singularities are fixed points and the remaining orbits are as large as possible, Grove and Searle \cite{Searlediff} found that the manifold is diffeomorphic to a CROSS - a compact rank one symmetric space (equivariance with a linear action was proved later in \cite{galaz2014note}). Such isometric actions are called \textit{fixed point homogeneous}.

The analogue to transitive group actions in the context of singular Riemannian foliations reveals an interesting difference between the two settings. Evidently, having a transitive group action (one with a single orbit) provides enough information to achieve a classification in positive curvature. However, \textit{any} manifold can be endowed with a single leaf foliation, so such foliations (when not induced by a group action) are trivial and provide no topological information. Nonetheless, one can still ask a similar question to that above. Namely, ``how large can the leaves of a singular Riemannian foliation be in the presence of point leaves, and what does this say about the manifold?'' We refer to such foliations as \textit{point leaf maximal}. We will prove the following,

\begin{mtheorem}
Let $M$ be a closed, connected, positively curved manifold admitting a point leaf maximal singular Riemannian foliation. If $M$ is simply connected, then $M$ is homeomorphic to a sphere or has the cohomology ring of a CROSS. If $M$ is not simply connected, then $M$ is either homeomorphic to a spherical space form or is a $\mathbb{Z}_2$-quotient of an odd dimensional cohomology complex projective space. 
\end{mtheorem}

In the non-simply connected case, we find examples not previously covered by the analogous results for isometric group actions. That is, we find examples of point leaf maximal SRFs that are not induced by fixed point homogeneous isometric group actions.

The layout of this paper is as follows: Section (\ref{Pre}) contains the basic definitions and tools used from the theory of singular Riemannian foliations (SRFs) and Alexandrov geometry that are used throughout. In section (\ref{PLM}), we define \textit{point leaf maximal} singular Riemannian foliations. Their structure is described in (\ref{Str}) and we discuss two submanifolds which play an important role in (\ref{sbmfd}). The classification is divided into cases depending on which of these submanifolds occur, which we study in (\ref{SC}) and (\ref{NSC}). Collecting the results, we obtain the main theorem.

I would like to thank my advisor Karsten Grove for his unending patience and invaluable insights. I would also like to thank Marco Radeschi for sharing his expertise on singular Riemannian foliations and Stephan Stolz for his help on all things topological.

\section{Prelimaries} \label{Pre}
\subsection{Singular Riemannian Foliations} \label{SRFs}
Here we provide a rapid rundown of the material used throughout. Everything in this section can be found with proofs in \cite{molino} (see chapter 6) and also in the unpublished lecture notes of Marco Radeschi \cite{srf}. See either source for a thorough treatment of the topics, as some more involved definitions have been omitted here for brevity.\\

A \textit{singular Riemannian foliation} is a partition $\FF=\{L_p\}_{p\in M}$ of a Riemannian manifold $(M,g)$ into smooth, connected, injectively immersed submanifolds (called \textit{leaves}) with two properties: (1) $\FF$ is a \textit{transnormal system} - i.e. any geodesic that emanates perpendicular to a leaf is perpendicular to all leaves it intersects and (2) $\FF$ is a \textit{singular foliation} - i.e. there exists a collection of vector fields on $M$ that span the tangent space to the leaves at all points. We refer to the triple $(M,g,\FF)$ as a singular Riemannian foliation (SRF) and may simply write $(M,\FF)$ when the metric on $M$ is understood. For our purposes, we will only consider singular Riemannian foliations (SRFs) whose leaves are closed.

Singular Riemannian foliations generalize orbit decompositions of a manifold by isometric group actions (which are a special case, called \textit{homogeneous SRFs}). As in the group action case, it is important to study the local, infinitesimal behavior about the leaves. For this, fix $p\in M$ and let $P$ be a connected open subset of $L_p$ containing $p$, sometimes called a \textit{plaque}. Next, let $\epsilon>0$ be sufficiently small so that $\exp:\nu^\epsilon L_p\to M$ is a diffeomorphism onto its image (here, $\nu^\epsilon L_p$ denotes the normal space to $L_p$ at $p$). A \textit{distinguished tubular neighborhood} is an open set $O\subset M$ of the form $O=O_{\epsilon}(P):=\text{exp}(\nu^{\epsilon}P)$. There is a partition $\FF_O$ of $O$ by the connected components of the intersections $L'\cap O$, for $L'\in \FF$. We use $\FF_O$ to form a singular Riemannian foliation on $T_pM$ with respect to the Euclidean metric.

Let $U\subset T_pP$ be a coordinate chart, and $(\nu_p^{\epsilon} P, \FF_p)$ be the foliation defined by pulling back the leaves of $(O,\FF_O)$ via $\exp:\nu_p^{\epsilon}P\to O$. There is a map $\Phi_O:(U\times \nu_p^{\epsilon}P, U\times \FF_p )\to (O,\FF_O)$ defined in terms of flows of linearized extensions of the coordinate vector fields which is a foliated diffeomorphism. Moreover, the rescaling map given by $r_\lambda(x,y)=(\lambda x, \lambda y)$ for $\lambda\in (0,1)$ is a foliated diffeomorphism of $(U\times \nu_pP, U\times \FF_p)$. Thus, we can pull back the metric from $O\subset M$ via the composition \[U\times \nu_pP\overset{r_\lambda}\longrightarrow U\times \nu_pP\overset{\Phi_O}\longrightarrow O\] and obtain a singular Riemannian foliation $(U\times \nu_p^\epsilon P,g',U\times \FF_p)$, where $g'=\left(\Phi_O\circ r_\lambda \right)^*g$. The same leaves also form a singular Riemannian foliation with the metric $g^\lambda:=\frac{1}{\lambda^2}g'=\frac{1}{\lambda^2}\left(\Phi_O\circ r_\lambda \right)^*g$. As $\lambda\to 0$, the metrics $g^{\lambda}$ converge pointwise to the flat metric $g_p$ on $U\times \nu_p^\epsilon P$. This SRF extends to an SRF on $T_pM$ which is invariant under rescalings and splits as $(T_pP\times \nu_pP, T_pP \times \FF_p)$. We will refer $(\nu_pP,g_p,\FF_p)$ as the \textit{(normal) infinitesimal foliation}. Because rescalings preserve $\FF_p$, we will also use this term to refer to the restriction of $\FF_p$ to the unit normal sphere $S_p^{\perp}$ to $L_p$ at $p$.

\subsection{Holonomy} \label{Hol}
The leaves of $(\nu_p P, \FF_p)$ are diffeomorphic to the intersections of the global leaves of $\FF$ with the normal slice $\exp_p(\nu_p^\epsilon L)$. In general, some of these `local' leaves will belong to the same global leaf and are identified via the \textit{holonomy group} $G_p\subset \textit{\textbf{O}}(\nu_pL_p)$, defined as the group of linearized lifts along $p$-loops (``sliding along a leaf'' in Molino \cite{molino}). Elements of $G_p$ in the identity path component $G_p^0$ will take a point in a local leaf to a point in the same local leaf. For such an element $g$, the existence of a continuous path  $g(s):[0,1]\to G_p$ with $g(0)=g$ and $g(1)=e$ means that for $v\in \nu_p L_p$, $g_s(v)$ traces out a continuous path in a local leaf (it's in a constant leaf of $\FF$ and stays in the slice through $p\in L_p$). Thus, at the level of leaves, we see that the group $G_p^0$ is the ineffective kernel of the action of $G_p$ on $O/\FF_O$. We will refer to the finite group $\Gamma_p:=G_p/G_p^0$ as the \textit{leaf holonomy group of $L_p$}, which acts effectively on the leaf space $O/\FF_O$. \\

Note that the map $\pi_1(L_p)\to \Gamma_p$ sending a homotopy class of loops to its path component in $G_p$ is a surjection. In particular, when $L_p$ is simply connected, the leaf holonomy is trivial. \\

\subsection{Stratification}
A singular Riemannian foliation can be naturally stratified by the dimension of the leaves. We say $\dim(\FF)=k
$ if the maximal dimension among the leaves of $\FF$ is $k$. Leaves of this maximal dimension $k$ are called \textit{regular} leaves. For $d\leq k$, define \[M_{(d)}=\{p\in M\ |\ \dim(L_p)=d\}\] The connected components of $M_{(d)}$ are called the \textit{$d$-strata} and each such component is called a \textit{$d$-stratum}. Each stratum is a submanifold of $M$ and the restriction of $\FF$ to a $d$-stratum is a regular Riemannian foliation. As with principal orbits of group actions, there is only one connected component of regular leaves (the \textit{regular stratum}) and it is open and dense. This stratification is known as the \textit{canonical} (or \textit{coarse}) \textit{stratification}.

Abusing the language a bit, we will sometimes also refer to the image of components of $M_{(d)}$ under the quotient map $M\to M/\FF$ as the $d$-strata and denote them by $\overline{M}_{(d)}$.

\begin{rmk}
There exists a finer stratification by further distinguishing leaves by their leaf holonomy, but we will not make use of this here.
\end{rmk} 

\subsection{Alexandrov Geometry} \label{AG}
The map $\pi:M\to M/\FF$ is an example of a \textit{submetry} (a map which takes metric balls onto metric balls of the same radius). Since lower curvature bounds are local properties and can be described in purely metric terms,  the target of a submetry from a manifold $M$ with $\text{sec}(M)\geq k$ is an Alexandrov space with the same lower curvature bound (in the distance comparison sense). In this section, we interpret some basic Alexandrov geometry concepts in the context of singular Riemannian foliations.  We refer the reader to \cite{BGP} for basic definitions and \cite{Grovevia} for the geometry of orbit spaces.\\

A basic notion in Alexandrov geometry is that of the \textit{space of directions}. The following natural result is what one might expect:

\begin{prop} The space of directions $\Sigma_{\bar{p}}$  at $\bar{p}=\pi(p)\in M/G$  consists of geodesic directions and is isometric to $\left(S_p^{\perp}/\FF_{p}\right)/\Gamma_p$, where $\FF_{p}$ is the restriction of the normal infinitesimal foliation to $S_p^{\perp}$.
\end{prop}

\begin{proof}
Identify $(O,g,\FF_O)$ with $(U\times \nu_p P,\Phi_O^*g, \FF_p)$ by pulling back both the foliation and the metric via $\Phi_O:U\times \nu_p P\to O$. Consider the action of $\Gamma_p=G_p/G_p^0$ on $O/\FF_O$ described as follows: take an element $g_\gamma\in \Gamma_p$ (represented by the `end-point map' of linearized lifts of normal vectors along the $p$-loop $\gamma$) and a leaf $L$ of $\FF_O$ (i.e. an element of $O/\FF_O$). Choose a point $q=\exp_p(v)$ in this leaf, so that $L=L_{\exp_p(v)}$. Then $L_{\exp_p(V(1))}$, where $V(t)$ is a linearized lift of $v$ along $\gamma$, is also a leaf of $\FF_O$, hence an element of $O/\FF_O$. So define $g_\gamma\cdot L_{\exp_p(v)}:= L_{\exp_p(V(1))}$

Since linearized lifts along loops carry leaves to leaves, our choice of $q\in L_{\exp_p(v)}$ is irrelevant at the level of leaves (provided $q$ is in the slice through $p$). So the above action is well defined (and effective, as we saw earlier). By equidistance of leaves of $\FF$ (and $\FF_O$) it follows that this action is by isometries.

Note that if $O$ is a radius $\epsilon$ tubular neighborhood of $P\subset L_p$, then the orbit space $(O/\FF_O)/\Gamma_p$ is isometric to an $\epsilon$-neighborhood of $\bar{p}\in M/\FF$. Moreover, since linearized lifts commute with homothetic transformations (lifts flow vertically along leaves while homotheties are horizontal), it follows that $\Gamma_p$ acts by isometries on the leaves of $(O,r_\lambda^* g, \FF_O)$ with quotient isometric to a $\lambda$-rescaling of the $\epsilon$-neighborhood above. Moreover, $\Gamma_p$ acts on the leaves $(O,g^\lambda, \FF_O)$, where $g^\lambda$ is as defined in section \ref{SRFs}. In the limit as $\lambda\to 0$, we obtain an action of $\Gamma_p$ on the tangent cone $T_{L_p}(O/\FF_O)$ with quotient isometric to the tangent cone $T_{\bar{p}}(M/\FF)$. Since the leaf space $O/\FF_O$ (for any of the rescaled metrics) consists only of normal directions (to $L_p$ at $p$), in the limit, we get an action of $\Gamma_p$ on the normal cone to $L_p$ whose orbit space is isometric to $T_{\bar{p}}(M/\FF)$. Since the normal cone upstairs is the cone on $(S_p^{\perp}/\FF_p)/\Gamma_p$, the tangent cone downstairs is the cone on $\Sigma_{\bar{p}}$, and all these directions are geodesic (by the slice theorem) the result follows. 
\end{proof}

\begin{rmk}
In his thesis, Marco Radeschi proved that Singular Riemannian foliations of spheres decompose as a join $(S_p^{\perp},\FF_p)\cong(SF_p,\FF_0)*(SF_p^{\perp},\FF_1)$, where $SF_p$ is a (totally geodesic) subsphere foliated by points and $\FF_1$ contains no point leaves. Both subspheres are invariant under the action of $\Gamma_p$ and we refer to $SF_p/\Gamma_p$ as the \textit{space of tangent directions} to the stratum of $\bar{p}$ and $(SF_p/\FF_1)/\Gamma_p$ as the \textit{space of normal directions} to the stratum.\\
\end{rmk}

Another important notion for Alexandrov spaces is that of the \textit{boundary}, which for an Alexandrov space $X$ is defined inductively as the collection of points for which the space of directions (itself an Alexandrov space) has boundary, \[\bdy X:=\{x\in X:\bdy\Sigma_x\neq\emptyset\}.\]

An extreme case is when the space of normal directions to the stratum of $\bar{p}$ is a single point. In such a case, the total space of directions $\Sigma_{\bar{p}}$ will have boundary $SF_p/\Gamma_p$. This happens in two cases:
\begin{enumerate}
\item When the infinitesimal foliation on the normal space to the stratum is by concentric spheres (i.e. when the foliation $\FF_p$ of $SF_p^{\perp}$ is by a single leaf).
\item When $SF_p^{\perp}=S^0$ and $\Gamma_p=\mathbb{Z}_2$. In which case, the infinitesimal foliation is still by concentric spheres (in a sense).
\end{enumerate}

\begin{definition}
Let $\bar{p}\in M/\FF$ be a point and $p\in \pi^{-1}(\bar{p})\subset M$. A connected components of the stratum of $\bar{p}$ is said to be an (open) \textit{boundary face} if the infinitesimal foliation of the normal space to the (dimension) stratum of $p$ is a foliation by concentric spheres .
\end{definition}

Intuitively, the boundary faces are those components of a stratum for which the only direction normal to the stratum is toward the interior. One may also define the open boundary faces as those components of strata for which the tangent cone splits as the tangent space to the stratum and a single normal ray. This hints at yet another definition of open boundary faces as the codimension 1 strata of an Alexandrov space.

As one might expect, these definitions of boundary face agrees with the usual one given for orbit spaces. That is, in the case of homogeneous SRFs, the boundary faces are those components of strata for which the isotropy action on the spheres normal to the stratum is transitive (see ~\cite{Grovevia}). As in that case, we also have

\begin{prop}
The boundary $\bdy (M/\FF)$ is the closure of the boundary faces.
\end{prop}

\begin{proof}
Clearly, a point in the closure of a boundary face cannot be interior. On the other hand, let $\bar{p}$ be an arbitrary point of $\bdy M/\FF$. Since nearby points in directions normal to the stratum of $\bar{p}$ are necessarily into less singular strata and the regular stratum is dense, there are directions normal to the stratum of $\bar{p}$ into the regular stratum.\\
If every direction normal to the stratum of $\bar{p}$ is toward the regular stratum, then there are no less singular boundary stratum. Thus, there is only one such normal direction and $\bar{p}$ belongs to a boundary face. Otherwise, at least one normal direction is into a less singular, nonregular stratum. Let $\bar{q}$ be a point close to $\bar{p}$ in such a direction. Then as above, either $\bar{q}$ belongs to a boundary face, or it does not, in which case, we continue on. Because there are only finitely many strata, this process ends in a boundary face. This shows that any neighborhood of $\bar{p}$ intersects a boundary face, completing the proof. 
\end{proof}

\section{Point Leaf Maximal SRFs} \label{PLM}
We will now generalize results from \cite{Searlecirc} and \cite{Searlediff} to the setting of singular Riemannian foliations, guided by the ideas and tools developed there. We begin with a general statement concerning the structure of positively curved leaf spaces with boundary. This is particularly useful when the boundary is smooth, as it implies that the only singularities occur in a boundary and possibly at the ``soul''. We then define a special type of SRF which ensures that this boundary is smooth and we see that manifolds admitting these special SRFs decompose as a union of disc bundles about ``nice'' submanifolds, which we then use to recover the cohomology of such manifolds.

\subsection{Structure} \label{Str}

\begin{lemma}[Stratum Lemma]
Let $(M,\FF)$ be a singular Riemannian foliation with closed leaves such that $M/\FF$ is positively curved with nonempty boundary $\bdy(M/\FF)$. Then 
\begin{enumerate}
\item[(i)] there is a unique point $\bar{s}\in M/\FF$ (the soul) at maximal distance from the boundary $\bdy M/\FF$.
\item[(ii)] the closure of any interior singular stratum contains both the soul point and points on the boundary, or is the soul point alone.
\end{enumerate}
\end{lemma}

\no Let $\bar{p}\in M/\FF$ be an interior, non-soul point. Let $\gamma(t)$ be a unit speed minimal geodesic from $\bar{p}$ to $\bar{s}$ and $\bar{x}\in \bdy M/\FF$ be a point realizing $\text{dist}(\bdy M/\FF,\bar{p})$.\\

\begin{claim}
The angle $\ang \bar{p}_{\bar{s}}^{\bar{x}}$ is strictly greater than $\pi/2$.
\end{claim}
\begin{proof}[Proof of claim]
Suppose otherwise and consider the comparison hinge (the configuration in $S^2$ with $|\tilde{p}\tilde{x}|=|\bar{p}\bar{x}|$, $|\tilde{p}\tilde{s}|=|\bar{p}\bar{s}|$, and $\ang \tilde{p}_{\tilde{s}}^{\tilde{x}}=\ang\bar{p}_{\bar{s}}^{\bar{x}}$) Let $\tilde{\gamma}(t)\in S^2$ be the unit speed geodesic from $\tilde{p}$ to $\tilde{s}$. We will derive a contradiction by analyzing and comparing the behavior at $t=0$ of the following functions:
\begin{align*}
d_{\bdy}(t) 			&:= \text{dist}(\bdy M/\FF,\gamma(t))\\
d_{\bar{x}}(t) 	    	&:= \text{dist}(\bar{x}, \gamma(t))\\
\tilde{d}_{\tilde{x}}(t)&:= \text{dist}(\tilde{x}, \tilde{\gamma}(t))
\end{align*}
It is easy to see that if $\ang \tilde{p}_{\tilde{s}}^{\tilde{x}}=\ang\bar{p}_{\bar{s}}^{\bar{x}}\leq \pi/2$, then $\frac{d \tilde{d}_{\tilde{x}}}{dt}|_{t=0}\leq 0$. Thus for any fixed positive real number $c$, there is an $\epsilon>0$ sufficiently small such that \[\frac{\tilde{d}_{\tilde{x}}(\epsilon)-\tilde{d}_{\tilde{x}}(0)}{\epsilon}<c.\] In particular, there exists such an $\epsilon$ so that \[\frac{\tilde{d}_{\tilde{x}}(\epsilon)-\tilde{d}_{\tilde{x}}(0)}{\epsilon}<b-a\] where $b=d_{\bar{x}}(1)=|\bar{s}\bar{x}|$ and $a=d_{\bar{x}}(0)=d_{\bdy}(0)=|\bar{p}\bar{x}|$.\\
By the hinge version of the Alexandrov comparisons, $\tilde{d}_{\tilde{x}}(t)\geq d_{\bar{x}}(t)$, and since $\bar{x}\in \bdy M/\FF$, $d_{\bar{x}}(t)\geq d_{\bdy}(t)$. So it follows that, \[\frac{d_{\bdy}(\epsilon)-d_{\bdy}(0)}{\epsilon}< b-a \leq d_{\bdy}(1)-d_{\bdy}(0)\] contradicting the convexity of $d_{\bdy}(t)$ along $\gamma$ and proving the claim.
\end{proof}

\begin{proof}[Proof of lemma]
The first statement follows from the strict concavity (along geodesics not minimal to the boundary) of the function $\text{dist}(\bdy M/\FF, \cdot)$ for positively curved Alexandrov spaces. For the second statement, we note that the claim implies that the space of directions $\Sigma_{\bar{p}}$ has diameter $>\pi/2$. Since $\Sigma_{\bar{p}}$ is itself isometric to the quotient of an SRF on a sphere (the infinitesimal foliation on normal sphere to $L_p$ at a point $p\in \pi^{-1}(\bar{p})$), it follows that this infinitesimal foliation is a suspension (see \cite{GMP}, Lemma 1.1). In particular, the infinitesimal foliation contains point leaves at distance $\pi$. Since point leaves in the normal space $\nu_pL_p$ represent leaves in the same dimensional stratum as $L_p$, it follows that the stratum of $\bar{p}$ extends toward both $\bdy (M/\FF)$ and $\bar{s}$. Since the closure of the stratum of $\bar{p}$ is compact, there is a (possibly non-unique) point $\bar{p}'$ in this closure nearest to $\bdy(M/\FF)$. Note that $\bar{p}'$ belongs to at least as singular a stratum as $\bar{p}$ (i.e. a stratum of leaves of lower dimension). The same argument shows that the closure of the stratum of $\bar{p}'$ also extends toward $\bdy M/\FF$. Because there are only finitely many strata, continuing in this manner will necessarily produce a point, the closure of whose stratum will contain points of $\bdy M/\FF$. Of course, these points also belong to the closure of the original stratum (that of $\bar{p}$), which shows that all interior strata ``stretch'' from $\bar{s}$ to $\bdy M/\FF$.
\end{proof}

As mentioned in the beginning of this section, when $\bdy(M/\FF)$ is a smooth manifold, the lemma gives that the only possible interior singularity is the soul point itself. Moreover, there are no critical points for $\text{dist}(\bdy(M/\FF,\cdot))$ in $M/\FF - (\bar{s}\cup \bdy(M/\FF))$. One can then use gradient-like vector fields to recover topological information about the manifold $M$. Guaranteeing that $\bdy(M/\FF)$ is a smooth manifold can be thus be viewed as one possible motivation for the following:

\begin{definition}
A singular Riemannian foliation with closed leaves is said to be \textit{point leaf maximal} if the infinitesimal foliation on the normal spheres to a (component of) the point leaf stratum are trivial, single leaf foliations.
\end{definition}

\no We denote by $\Sigma_0$ the set of point leaves of $(M,\FF)$ and by $F_0\subset \Sigma_0$, a component of $\Sigma_0$ which meets the condition of the definition above.

\begin{rmk}
By definition, $F_0$ (as a subset of $M/\FF$, into which it embeds isometrically) is a codimension 1 stratum of the Alexandrov space $M/\FF$. Since $\bdy (M/\FF)$ is the closure of the codimension 1 strata, it follows that $F_0$ is a boundary component. Also, when $\dim(M/\FF)>1$ and $M/\FF$ is positively curved, the boundary is connected. So, with the exception of codimension 1 SRFs, it is convenient to think of point leaf maximal foliations as those for which $F_0=\partial M/\FF$. In the case that $\dim(M/\FF)=1$, $M/\FF$ is an interval and $F_0$ is one of the two endpoints. 
\end{rmk}

\begin{theorem}[Structure Theorem]
Let $(M,\FF)$ be a point leaf maximal singular Riemannian foliation such that $M/\FF$ is positively curved. If $F_0$ is a component of $\Sigma_0$ with maximal dimension, then the following hold:
\begin{enumerate}
\item[(i)] There is a unique `soul' leaf, $L_s$, at maximal distance to $F_0$.
\item[(ii)] All leaves in $M-\left(F_0\cup L_s\right)$ are principal and diffeomorphic to $S^k$, the normal sphere to $F_0$.
\item[(iii)] The infinitesimal foliation $\left(S^{\ell},\nu_s\FF_|\right)$ of the normal sphere $S^\ell$ to $L_s$ at $s$ is a (principal) Riemannian foliation. Moreover, $F_0$ is diffeomorphic to $\left(S^\ell/\nu_s\FF_|\right)/\Gamma_s$, the space of directions at $\bar{s}$.
\item[(iv)] $M/\FF$ is homeomorphic to a cone on $F_0$. 
\item[(v)] $M$ is foliated diffeomorphic to the union of two normal disc bundles $D(F_0), D(L_s)$ (viewed as tubular neighborhoods in $M$) glued along their common boundary.
\end{enumerate}
\end{theorem}

\begin{proof}
Statement (i) carries over from the stratum lemma, viewed ``upstairs'' in $M$. Statement (ii) follows from the stratum lemma and the definition of point leaf maximal. The first part of statement (iii) follows from statement (ii). The remaining statements are now proved altogether by constructing a gradient-like vector field for $\text{dist}(\bar{s},\cdot)$ which is radial near both $\bar{s}$ and $\bdy(M/\FF)\cong F_0$. With a suitable rescaling, we arrive at a vector field on $M/\FF$, the flow of which provides the both the diffeomorphism of (iii) and the homeomorphism of (iv). One sees statement (v) by lifting this vector field horizontally to $M$ and removing tubular neighborhoods about the soul leaf $L_s$ and the $F_0$, giving a manifold (in $M$) with two boundary components and a gradient-like vector field radial to both of them (a Morse isotopy type argument).\\  

\no Now take a gradient-like vector field for $\text{dist}(\bar{s},\cdot)$ on $M/\FF$ (call it $X$) and consider its restriction to \[\widehat{M}:=\{\bar{p}\in M/\FF\ |\ \text{dist}(\bar{s},\bar{p})\geq \epsilon\}\] 
$\widehat{M}$ is a manifold with two (smooth) boundary components, one of which $\bdy M/\FF=F_0$ and the other is the boundary of a small metric ball about $\bar{s}$, which is isometric to space of directions $\left(S^l/\nu_s\FF_|\right)/\Gamma_s$ (see the first proposition of \ref{AG}). Since $\text{dist}(\bar{s},\cdot)$ has no critical points in $\widehat{M}$ and $X$ is clearly radial to both boundary components, the flow this vector field gives the diffeomorphism in the second part of statement (iii). In fact, taking the flow all the way to $\bar{s}$, we get statement (iv).\\

\no For the last statement, we further restrict the gradient-like vector field $X$ to \[\widehat{M}':=\{\bar{p}\in M/\FF \ | \ \text{dist}(\bar{s},\bar{p})\geq \epsilon \ \& \ \text{dist}(\bdy(M/\FF),\bar{p})\geq \epsilon\}\] For sufficiently small $\epsilon$, $X$ is still radial to both its boundary components. Note that the unique horizontal lift of $X$ to $\pi^{-1}(\widehat{M}')$ is gradient-like for $\text{dist}(L_s,\cdot)$ and still critical point free. Moreover, $\pi^{-1}(\widehat{M}')$ has two boundary components, which are the (smooth) boundaries of tubular neighborhoods of $L_s$ and $F_0$, respectively. The flow of this lift gives the diffeomorphism of (v). That this diffeomorphism is foliated follows from the homothetic transformation lemma and the definition of point leaf maximal.
\end{proof}

\begin{rmk}
We assume above that $M/\FF$ is positively curved (which guarantees the soul and that $F_0\cong \bdy M/\FF$). However, this doesn't make sense when $\dim(M/\FF)=1$. In that case, one must decide what is meant by the ``soul.'' Certainly, $F_0$ will be one of the endpoints of the interval $M/\FF$. If we call the other endpoint the soul, then everything follows as above.
\end{rmk}

\subsection{The Submanifolds $L_s$ and $F_0$} \label{sbmfd}
For a point leaf maximal SRF $(M,\FF)$, the submanifolds $L_s$ and $F_0$ are in some sense ``dual'' in $M$. We wish to classify what these two submanifolds can be, then, given the decomposition $M=D(L_s)\cup D(F_0)$, determine which positively curved manifolds $M$ can appear. To begin, we exhibit fiber sequences in which $L_s$ and $F_0$ appear.\\

Recall that the leaves of the infinitesimal foliation on $S^\ell$ are (diffeomorphic to) the intersections of the global leaves of $\FF$ with $\exp_s(S^\ell)$. Explicitly, they are obtained by restricting to $S^\ell$ the pullback foliation by the map $\exp:\nu_sL_s\to O$, where $O$ is a distinguished tubular neighborhood of $L_s$. In the case of point leaf maximal SRFs, these leaves are all principal and the generic leaf is $\mathcal{L}=\exp_s^{-1}(\exp_s\left(S_{\epsilon}^{\ell})\cap L_p\right)$, where $L_p$ is a nearby leaf at fixed distance $\epsilon$ from $L_s$. Moreover, $S^\ell/\nu_s\FF_|$ (or $S^\ell/\FF_|$ for short) is a manifold with quotient map a Riemannian submersion, giving the following fibration \[
\mathcal{L}\longrightarrow S^\ell\longrightarrow S^\ell/\FF_|\]

\no Riemannian submersions from Euclidean spheres are classified up to metric congruence: they are exactly the Hopf fibrations (see \citep{gromoll1988low}, \citep{wilking2001index}). In particular, for nontrivial such submersions, $\mathcal{L}$ is a standard 1-, 3-, or 7-sphere and $S^\ell/\FF_|$ is  $\mathbb{CP}^n, \mathbb{HP}^n$ or $\mathbb{OP}^2$ with their standard Fubini-Study metrics.\\

The quotient of the \textit{local quotient} $S^\ell/\FF_|$ by the holonomy action of $\Gamma_s$ identifies the leaves of $\nu_s\FF_|$ belonging to the same global leaf, and is isometric to the space of directions at $\pi(L_s)\in M/\FF$ (which is in turn diffeomorphic to $F_0$). So in the case that $\Gamma_s$ acts trivially, we have
\begin{equation} \label{l}
\mathcal{L}\longrightarrow S^\ell \longrightarrow F_0 \tag{$\ell$}
\end{equation}
\no We address the case when $\Gamma_s$ acts nontrivially in section \ref{NSC}.\\

On the other hand, the closest point projection map from a leaf $L_p$ (at distance $\epsilon$ from $L_s$) to $L_s$ is a submersion (see \cite{srf}). The fiber of this map is clearly $\exp_s(S_{\epsilon}^{\ell})\cap L_p\cong \mathcal{L}$. This gives a fibration 
\begin{equation}\label{k}
\mathcal{L}\longrightarrow S^k\longrightarrow L_s \tag{$k$}
\end{equation}

\begin{rmk}
It would be nice to know what conditions imply trivial leaf holonomy about $L_s$. We saw in \ref{Hol} that this is guaranteed when $L_s$ is simply connected. If we insist that the ambient manifold $M$ is also simply connected, then by transversality, $\text{codim}(F_0)\geq 3$ implies $L_s$ is simply connected. But $\text{codim}(F_0)=0$ and $\text{codim}(F_0)=1$ imply that $\FF$ is a trivial (by points) foliation. So if $M$ is simply connected, nontrivial leaf holonomy about $L_s$ is only possible when $\text{codim}(F_0)=2$, in which case the regular leaves are necessarily circles. This forces $\dim(L_s)$ to be either 0 or 1. If $\dim(L_s)=0$, then the leaves of the foliation on $S^\ell$ are diffeomorphic to the global leaves of $\FF$ and the holonomy action is necessarily trivial. Thus, the only possible case where nontrivial leaf holonomy about $L_s$ may occur (while $M$ is simply connected) is when $L_s$ is an exceptional circle leaf. In such a point leaf maximal SRF, the infinitesimal foliations of both the principal leaves and $L_s$ are trivial (by points), and the infinitesimal foliations of the point leaves are products of a point foliation and a concentric sphere foliation. In particular $\FF$ is \textit{infinitesimally polar}. Thus, by theorem 1.8 in \cite{lytchak2010geometric}, we see that $L_s$ cannot possibly be exceptional. In short, if $(M,\FF)$ is point leaf maximal and $M$ is simply connected, there is no leaf holonomy. 
\end{rmk}

\subsection{The Simply Connected Case} \label{SC}
\subsubsection{Circle Leaves}
We assume now that $(M,\FF)$ is a point leaf maximal SRF and $(M,g)$ is a simply connected, positively curved Riemannian manifold. As mentioned at the end of the previous section, this guarantees that $L_s$ is simply connected unless $\text{codim}(F_0)=2$, in which case the generic leaves are circles. In this case, we have the following reformulation of Thm 1.2 in \cite{Searlecirc}: 

\begin{theorem}
Let $M$ be a simply connected positively curved manifold. If $M$ admits a 1 dimensional point leaf maximal singular Riemannian foliation, then $M$ is foliated diffeomorphic to a sphere $S^n$ or a complex projective space $\mathbb{CP}^m=S^{2m+1}/S^1$.
\end{theorem}

\begin{proof}
First note that such foliations are homogeneous (Thm 3.11 in \cite{galaz2015singular}), but possibly with respect to a different metric. That is, there is a metric $g'$ (not necessarily of positive curvature) and an isometric $S^1$ action on $(M,g')$ such that the orbits coincide with the leaves of $\FF$. In particular, $M/S^1$ and $M/\FF$ have diffeomorphic strata and isometric spaces of directions.

\no Now the dimension of $L_s$ is either 0 or 1, and since there is no leaf holonomy when $M$ is simply connected, it follows that $L_s$ is either a principal circle leaf or a point leaf. Fix a point $s\in L_s$ and a small metric normal sphere $S_{\epsilon}^\ell$ centered at $s$.\\ 

\no Suppose $L_s$ is a principal circle leaf. Then $M/\FF$ is a $(\ell+1)$-disc and $F_0\cong S^\ell$. Flowing the normal sphere above with the horizontal lift of a gradient-like vector field for $\text{dist}(\bar{s},\cdot)$, $\bar{s}\in M/\FF$, will cut out a smooth, nonvanishing section of $D(F_0)$, viewed as an $S^1$-bundle over $F_0$. 
The isometric action (with respect to $g'$) of $S^1$ on $M$ orients this bundle fiberwise, making it a principal $S^1$-bundle with a nowhere vanishing section, hence trivial. Flowing the normal spheres over all points of $L_s$ simultaneously then defines a diffeomorphism between $M-\left(\text{int}D(L_s)\cup \text{int}D(F_0)\right)$ and $S^1\times S^\ell \times (\delta,\pi/2-\delta)$ for some small $\delta>0$ (i.e. the ``interior'' of the join $S^1*S^{\ell}\cong S^{\ell+2}$). Since both $D(L_s)$ and $D(F_0)$ are trivial bundles, both normal projections are trivial (as are those to the focal spheres in the join) and we see that $M$ is diffeomorphic to $S^{\ell+2}$. This diffeomorphism is foliated if we foliate the sphere by the $S^1$'s from the first factor.\\

\no Suppose $L_s$ is a point leaf. Then $S^\ell=S^{2m+1}$ and $M/\FF$ is a cone on $F_0\cong \mathbb{CP}^m$. As above, flowing the horizontal lift of a gradient-like vector field for $\text{dist}(\bar{s},\cdot)$  will give a diffeomorphism between $M-\left(\text{int}D(L_s)\cup \text{int}D(F_0)\right)$ and $S^{2m+1}\times \{pt\} \times (\delta, \pi/2-\delta)$. Invoking the group action, which must be free on $M-\left(\text{int}D(L_s)\cup \text{int}D(F_0)\right)$, we see that the restriction $\pi: \bdy D(F_0)\to \bdy D(\pi(F_0))$ is a Riemannian submersion from $S^{2m+1}$ to $\mathbb{CP}^m$, i.e. it is the Hopf map. This is exactly the normal projection from $\bdy D(F_0)$ to $F_0$. Since the normal projection from $\bdy D(L_s)\cong S^{2m+1}$ to $L_s=\{pt\}$ is necessarily trivial, we see that $M$ is diffeomorphic to $\mathbb{CP}^m$. This diffeomorphism is foliated if we foliate $\mathbb{CP}^m$ by distance spheres around the fixed point corresponding to $L_s$.
\end{proof}

\subsubsection{Higher Dimensional Leaves}
\no Now we will assume that $\pi_1(M)=0$ and $\text{codim}(F_0)\geq 3$. Recall that in these cases, we have the fiber sequences \eqref{l} and \eqref{k}:
\begin{equation}
\mathcal{L}\longrightarrow S^\ell \longrightarrow F_0 \tag{$\ell$}
\end{equation}
\begin{equation} 
\mathcal{L}\longrightarrow S^k \longrightarrow L_s \tag{$k$}
\end{equation}

\no As we mentioned previously, as long as \eqref{l} is a nontrivial fibration (i.e. $\mathcal{L}\neq \{pt\}$ and $\mathcal{L}\neq S^\ell$), then the fiber  $\mathcal{L}$ is isometric to $S^1, S^3$ or $S^7$ (and $S^7$ can only occur when $\ell=15$). Moreover, the space of directions at the soul of $M/\FF$ (call it $\Sigma$) is isometric to a compact rank one symmetric space (CROSS). Thus, $F_0$ is diffeomorphic to a CROSS unless $\mathcal{L}=S^\ell$, in which case $F_0$ is a point. On the other hand, because the sphere $S^k$ is NOT a standard round sphere, the possibilities for $L_s$ are the same, but only up to cohomology ring (this can be seen via the Gysin sequence, for example). This is a key difference between point leaf maximal SRF's and fixed point homogeneous group actions (where both submanifolds are known up to diffeomorphism).\\

Cases can now be generated by the topological restrictions imposed by the fibrations \eqref{l} and \eqref{k}. Based on the dimension of $\mathcal{L}$, we get the following possibilities for $F_0$ (up to diffeomorphism) and $L_s$ (up to cohomology ring):

\begin{center}
\begin{tabular}{|c |c | c | c| c| c|}
\hline
$\text{Case}$	& $dim(\mathcal{L})$ & $S^l$ 	& $F_0$ 	& $S^k$ & $L_s$\\
\hline
\hline
$(A)$ & $l=k$	& $S^k$	  & $\{pt\}$ & $S^k$ 	  & $\{pt\}$\\
$(B)$ & $0$ & $S^\ell$ 	  & $S^\ell$		& $S^k$ 	 & $S^k$\\
$(C)$ & $1$ & $S^{2m+1}$  & $\mathbb{CP}^m$ & $S^{2j+1}$ & $\mathbb{CP}^j$ \\
$(D)$ & $3$ & $S^{4m+3}$  & $\mathbb{HP}^m$	& $S^{4j+3}$  & $\mathbb{HP}^j$  \\
$(E)$ & $7$ & $S^7, S^{15}$	  & $\{pt\}, S^8$	 & $S^{15}, S^{7}$	  & $S^8,\{pt\}$ \\
$(X)$ & $7$ & $S^{15}$ & $S^8$ & $S^{15}$ & $S^8$ \\
\hline
\end{tabular}
\end{center}

\begin{rmk}
Evidently, the larger $F_0$ is relative to $L_s$, the more we can say about $M$. This is particularly interesting when $L_s=\{pt\}$ (where we get a homeomorphism classification of $M$). However, the arguments below are purely algebraic topological and ``symmetric'' in $F_0$ and $L_s$, so we will assume without loss of generality that $\dim(F_0)\leq \dim(L_s)$ (or equivalently, $\ell\leq k$).
\end{rmk}

\no Cases \textbf{(A)} and \textbf{(B)} correspond to the two ``trivial'' fibrations \eqref{k} and \eqref{l}. We address case \textbf{(A)} first:

\begin{theorem}
Let $M$ be a simply connected positively curved manifold. If $M$ admits a codimension 1 point leaf maximal SRF such that the soul leaf is a point, then $M$ is foliated homeomorphic to a sphere.
\end{theorem}

\begin{proof}
In this case, we have that the endpoints of $M/\FF$ correspond to two isolated point leaves of $\FF$. Thus, $M$ is two $(k+1)$-discs foliated by concentric $S^k$ glued along their common boundary $S^k$. This is clearly foliated homeomorphic to $S^{k+1}$ (foliated by `lateral' subspheres).
\end{proof}

\no Next up is case \textbf{(B)}.

\begin{theorem}
Let $M$ be a simply connected positively curved manifold. If $M$ admits a point leaf maximal SRF such that the infinitesimal foliation on the normal sphere of the soul leaf is by points, then $M$ is foliated homeomorphic to a sphere.
\end{theorem}

\begin{proof}
Although we do not use it, note that in this case $M/\FF$ is homeomorphic to an $(\ell+1)$-disc. This is because $L_s$ in this case is itself a principal leaf, so $M/\FF$ is a smooth manifold (with boundary) whose boundary is isometric to a sphere $F_0\cong S^\ell$.\\
Because $L_s \cong S^k$ is principal, we have $D_{\epsilon}(L_s)=S^k\times D^{\ell+1}$, foliated by $S^k$'s from the first factor. As a model space, consider the join foliation $(S^\ell,\{pts\})*(S^k,S^k)$ on $\widehat{M}=S^{\ell+k+1}$. This is a point leaf maximal SRF, with $\widehat{L_s}=S^k$ and $\widehat{F_0}=S^\ell$. The tubular neighborhood about $\widehat{L_s}$ is $S^k\times D^{\ell+1}$, as in our case. Clearly, this is homeomorphic to $D(L_s)$, and in particular, the restriction to the boundary of this tubular neighborhood $S^k\times S^\ell$ is also a homeomorphism. We can then use the flows on $\widehat{M}$ and $M$ of the gradient-like vector fields for $\text{dist}(\widehat{L_s},\cdot)$ and $\text{dist}(L_s,\cdot)$, respectively, to uniquely extend this to a homeomorphism from $\widehat{M}$ to $M$. This homeomorphism a foliated one if one foliates $\widehat{M}=S^\ell*S^k$ by the $S^k$'s from the second factor .
\end{proof}

\no The following more general statement has a weaker conclusion than the results above, so it suffices to prove it for the remaining cases.

\begin{theorem}
Let $M^n$ be a simply connected positively curved manifold. If $M$ admits a point leaf maximal SRF, then $M$ has the cohomology ring of a CROSS. Moreover, if the soul leaf is a point, then $M$ is homeomorphic to a CROSS.
\end{theorem}

\begin{proof} 
Consider the following long exact sequence in relative cohomology:
\[...\rightarrow H^{i-1}(D(L_s))\longrightarrow H^i(M,D(L_s))\longrightarrow H^i(M)\longrightarrow H^i(D(L_s))\longrightarrow H^{i+1}(M,D(L_s))\rightarrow...\]
Now $H^i(M, D(L_s))\cong \widetilde{H}^i(M/D(L_s))$, and given that $M=D(F_0)\bigcup_E D(L_s)$, we see that $M/D(L_s)\simeq D(F_0)/S(F_0)$, where $S(F_0)$ denotes the $k$-sphere bundle over $F_0$. This is precisely the Thom space of the rank $k+1$ vector bundle over $F_0$ which induces $D(F_0)$. By the Thom isomorphism theorem, we have that \[H^i(M,D(L_s))\cong H^{i-(k+1)}(F_0)\] and since $D(L_s)$ deformation retracts to $L_s$, we have the long exact sequence \[...\rightarrow H^{i-1}(L_s)\longrightarrow H^{i-(k+1)}(F_0)\longrightarrow H^i(M)\longrightarrow H^i(L_s)\longrightarrow H^{(i+1)-(k+1)}(F_0)\rightarrow...\]

\no For dimension reasons, we recover the following cohomology groups:

\begin{enumerate}
\item[\textbf{(C)}] $M$ has additive cohomology of $\mathbb{CP}^{\frac{1}{2}(l+k)}$
\item[\textbf{(D)}] $M$ has additive cohomology of $\mathbb{HP}^{\frac{1}{4}(l+k-2)}$
\item[\textbf{(E)}] $M$ has additive cohomology of $\mathbb{OP}^2$ 
\item[\textbf{(X)}]$
H^i(M) = \left\{
        \begin{array}{ll}
            \mathbb{Z} & \quad i=0,8,16,24 \\
            0 & \quad \textnormal{otherwise}
        \end{array}
    \right.
$
\end{enumerate}

\no In the cases where $l<k$, the ring structure is recovered via Poincar\`e duality: the isomorphisms $H^i(M)\cong H^i(L_s)$ for $i\leq k-1$ induced by the inclusion $L_s\hookrightarrow M$ induce a ring isomorphism up to $i=\dim(L_s)$. In cases \textbf{(C)-(E)}, $\dim(L_s)\geq\frac{1}{2}\dim(M)$, so all cup products are determined. \\

\no For example, in case \textbf{(C)}, we know that up to $i=k-1$, the cohomology ring is generated by an element $[\alpha]\in H^2(M)$. Since $k-1\geq \frac{1}{2}\dim(M)$, Poincar\`e duality gives that the top cohomology $H^n(M)$ can be generated by cup products of generators below dimension $k-1$, which are powers of $[\alpha]$. It necessarily follows that all cohomology groups are generated by powers of $[\alpha]$.

\begin{rmk}
At this point, most cases are settled, but special care needs to be taken when $l=k$ since Poincar\'e duality fails to determine all cup products. We will exhibit the specifics in case \textbf{(C)} when $\ell=k$. The proof works the same way in the other cases, with minor adjustments.\\
\end{rmk}

\no So we have $l=k=2m+1$, $F_0=\mathbb{CP}^m, L_s\sim \mathbb{CP}^m$, and we know $M$ has the additive cohomology of $\mathbb{CP}^k$. As we saw above, we have a ring isomorphism between $H^*(L_s)$ and $H^*(M)$ up to dimension $i=k-1=2m$. So we have a element $[\alpha]\in H^2(M)$ for which $[\alpha]^m$ generates $H^{k-1}(M)$. Poincar\'e duality only tells us that there exists a generator of $H^{k+1}(M)$ whose cup product with $[\alpha]^m$ generates the top cohomology, but it does not guarantee that this generator is equal to $[\alpha]^{m+1}$. For this, we will exhibit a topologically embedded submanifold $N\cong \mathbb{CP}^{m+1}$ whose inclusion induces a ring isomorphism with $H^*(M)$ up to dimension $i=2m+2=k+1\geq \frac{1}{2}\dim(M)$. Then apply Poincar\'e duality as before. \\

\no Consider a small minimal geodesic segment emanating from $\pi(L_s)\in M/\FF$. The other endpoint of this segment can be flowed along the gradient-like vector field for $\text{dist}(F_0,\cdot)$ (which is radial near $F_0$) to a nearest point $p\in F_0$. This gives an integral curve $\gamma$ in $M/\FF$ which can be `covered' by two half curves ($\gamma_1$ from $p$ to the midpoint $x\in\gamma$, and $\gamma_2$ from $x$ to $\pi(L_s)$). Thus $\pi^{-1}(\gamma)$ is covered by the preimages of these two half segments. The preimage $\pi^{-1}(\gamma_1)$ is a $(2m+2)$-disc centered at $p$ and normal to $F_0$, whereas $\pi^{-1}(\gamma_2)$ is a subbundle of $D(L_s)$ (the fiber being the 2-disc bounded by a smoothly varying $\mathcal{L}$ at each point). In fact, the boundary of this subbundle is simply a global leaf $S^k$ (it is the preimage of the midpoint $x\in M/\FF$). This gives a submanifold $N\subset M$ whose cohomology can be attained from Mayer-Vietoris as
\begin{center}
\begin{tikzcd}[cells={nodes={minimum height=2em}}]
...\arrow[r] & H^{i-1}(S^k) \arrow[r] & H^i(N) \arrow[r] & H^i(\{pt\})\oplus H^i(L_s) \arrow[r] & H^i(S^k) \arrow[r] &...
\end{tikzcd}
\end{center}

\no from which we see that $N$ has the cohomology ring of $L_s\sim \mathbb{CP}^m$ up to $i=k-1$, for dimension reasons. Since $\dim(N)=k+1=2m+2$, it follows from Poincar\`e duality that $N$ has the cohomology ring of $\mathbb{CP}^{m+1}$. We have inclusions $N\hookrightarrow M$ and $(N, D(L_s)\cap N)\hookrightarrow (M, D(L_s))$ and thus a sequence of maps

\adjustbox{scale=.85,center}{
\begin{tikzcd}[cells={nodes={minimum height=2em}}]
\arrow[r] & H^{i-1}(L_s)\arrow[r]\arrow[d,"\cong"] & H^i(M,D(L_s)) \arrow[r]\arrow[d,"j^*"] & H^i(M) \arrow[r]\arrow[d,"i^*"] & H^i(L_s)\arrow[r]\arrow[d,"\cong"] &  H^{i+1}(M,D(L_s)) \arrow[r]\arrow[d,"j^*"] & \  \\
\arrow[r] & H^{i-1}(L_s)\arrow[r] & H^i(N,D(L_s)\cap N) \arrow[r] & H^i(N) \arrow[r] & H^i(L_s)\arrow[r] &  H^{i+1}(N,D(L_S)\cap N) \arrow[r] & \  
\end{tikzcd}
}

\no If we can show that the induced map $H^*(M,D(L_s))\overset{j^*}\to H^*(N,D(L_s)\cap N)$ is a ring isomorphism up to dimension $i=k+2$, we are done, by the five lemma. Now the map $j^*$ fits in the following diagram (``unpacking'' the Thom isomorphism)

\begin{center}
\begin{tikzcd}[cells={nodes={minimum height=2em}}]
H^i(M,D(L_s))\arrow[d,"j^*"] & H^i(D(F_0),\bdy D(F_0)) \arrow[l, "\text{exc}-L_s"', "\cong"]\arrow[d,"\eta^*"] & H^{i-(k+1)}(D(F_0)) \arrow[l, "\text{Thom}"',"\cong"]\arrow[d,"\eta^*"]\\
H^i(N,D(L_s)\cap N) & H^i(D^{k+1},S^k) \arrow[l, "\text{exc}-L_s"', "\cong"] & H^{i-(k+1)}(D^{k+1}(\{pt\})) \arrow[l, "\text{Thom}"',"\cong"]
\end{tikzcd}
\end{center} 

\no where the rightmost $\eta^*$ is technically a restriction of the middle $\eta^*$, both of which are induced by the inclusion $(D^{k+1}(\{pt\}),S^k)\overset{\eta}\hookrightarrow (D(F_0),\bdy D(F_0))$. Now since $H^{i-(k+1)}(F_0)=0=H^{i-(k+1)}(\{pt\})$ for $i\leq k$ and $i=k+2$, we need only check that the diagram above commutes for $i=k+1$ ($j^*$ is certainly an isomorphism when its domain and target are both $0$). The rightmost $\eta^*$ is an isomorphism between $H^0(D(F_0))$ and $H^0(D^{k+1}\{pt\})$, since both are path connected. Now if $\tau\in H^{k+1}(D(F_0),\bdy D(F_0))$ is the Thom class of $D(F_0)$, its restriction $\eta^*(\tau)\in H^{k+1}(D^{k+1},S^k)$ is the Thom class of $D^{k+1}(\{pt\})$ and since the Thom isomorphism is the cup product with the Thom class, the right square commutes by naturality of the cup product.\\

\no Thus, we conclude that the induced map $H^*(M,D(L_s))\overset{j^*}\to H^*(N,D(L_s)\cap N)$ is a ring isomorphism up to dimension $i=k+2$, as desired. Now Poincar\'e duality gives the correct cohomology ring $H^*(M)\cong\mathbb{Z}[x]/x^{k+1}$ in case \textbf{(C)} when $\ell=k$. Similarly for case \textbf{(D)} when $\ell=k$. This also shows that in case \textbf{(X)}, $H^*(M)\cong \mathbb{Z}[x]/x^4$. But in this case, $|x|=8$, which is impossible (see Corollary 4.L.10 in \cite{hatcher2002algebraic}). So case \textbf{(X)} can be thrown out. This proves the first statement.\\

\no In the special case that $L_s=\{pt\}$, then $D(L_s)$ is homeomorphic to a disc foliated by concentric $S^k$'s. The boundary of this disc is homeomorphic to a regular leaf $S^k$. This homeomorphism has a unique radial extension ``down to'' $F_0$, which is foliated as well, by the homothetic transformation lemma and the definition of point leaf maximal. This describes a foliated homeomorphism from $M$ to a CROSS of the same `type' as $F_0$ (of one higher dimension).
\end{proof}  

\subsection{The Non-simply Connected Case} \label{NSC}

\no If $(M,\FF)$ is a point leaf maximal SRF and $M$ is not simply connected, we can pullback the foliation to the universal cover $\tilde{M}$ and use the results above. We have

\begin{theorem}
If $M$ is a non-simply connected, positively curved manifold which admits a point leaf maximal SRF, then $M$ is a $\mathbb{Z}_2$-quotient of an odd dimensional cohomology complex projective space or is homeomorphic to a space form $S^n/\Gamma$. Moreover, the leaf holonomy group about $L_s$ is isomorphic to $\pi_1(M)$.
\end{theorem}

\begin{proof}
Let $(M,\FF)$ be our point leaf maximal SRF and consider its pullback foliation to the universal cover $(\wtilde{M},\wtilde{\FF})$. The condition of being point leaf maximal is a local one, hence holds for $(\wtilde{M},\wtilde{\FF})$. So we know from the work above that $\wtilde{M}$ is homeomorphic to a sphere or has the cohomology ring of a CROSS. \\

\no Denote by $\wtilde{L}_s$ and $\wtilde{F}_0$ the soul leaf and point leaf component (that fits the definition) for $\wtilde{\FF}$. Because $\wtilde{\FF}$ is the pullback foliation of $\FF$, it follows that \[(\wtilde{M}/\wtilde{\FF})/\pi_1(M)=M/\FF\]  and in particular, $\wtilde{F}_0/\pi_1(M)=F_0$ and $\wtilde{L}_s/\pi_1(M)=L_s$. Moreover, $\pi_1(M)$ is finite (Bonnet-Myers) and acts by isometries on the space of directions at the soul of the leaf space $\wtilde{M}/\wtilde{\FF}$, call it $\wtilde{\Sigma}$. Since we've shown previously that $\wtilde{\Sigma}$ is isometric to a CROSS, it follows that $\wtilde{\Sigma}$ is either a sphere or an odd-dimensional complex projective space. Since $\pi_1(M)$ must also act on $\wtilde{M}$, we must have that $\wtilde{M}$ itself is homeomorphic to a sphere or is an odd-dimensional cohomology complex projective space. This proves the first statement.\\

\no In the case that $\wtilde{M}$ is a sphere, we've seen that the infinitesimal foliation on the normal spheres to $\wtilde{L}_s$ is by points (and the same is true in $M$). Thus, $\wtilde{\Sigma}=S^\ell/\wtilde{\Gamma}$, where $\wtilde{\Gamma}$ is the leaf holonomy group about $\wtilde{L}_s,$ which is trivial since $\wtilde{M}$ is simply connected. Also, $\wtilde{\Sigma}/\pi_1(M)=\Sigma$, the space of directions at the soul of $M/\FF$, which is itself isometric to $S^\ell/\Gamma$, where $\Gamma$ is the leaf holonomy about $L_s$. So we realize $\Sigma$ in two ways: \[S^\ell/\pi_1(M)\cong \Sigma \cong S^\ell/\Gamma\]
Since both actions are free, we have that $\Gamma\cong \pi_1(M)$.\\

\no When $\wtilde{M}$ is an odd-dimensional cohomology complex projective space, we know that the infinitesimal foliation of $S^\ell$ (in both $\wtilde{M}$ and $M$) is the Hopf fibration. Again, because $\wtilde{M}$ is simply connected, we know that $\wtilde{\Gamma}=\{1\}$, so $\wtilde{\Sigma}$ is isometric to $\mathbb{CP}^n$ for some odd $n$. Since the only finite free action on odd-dimensional complex projective space is by $\mathbb{Z}_2$, we know that $\pi_1(M)\cong \mathbb{Z}_2$. But $\Sigma$ is also a metric quotient of $\mathbb{CP}^n$ by the finite group $\Gamma$, so $\Gamma\cong \pi_1(M)\cong \mathbb{Z}_2$.
\end{proof}

We get the following special result when the leaves are circles, whose proof is the same as above, but with stronger conclusion coming from the stronger conclusion in the simply connected case:
\begin{theorem}
If $M$ is a non-simply connected, positively curved manifold admitting a 1-dimensional point leaf maximal SRF, then $M$ is diffeomorphic to a space form $S^n/\Gamma$.
\end{theorem}

\begin{rmk}
The only finite quotients of spheres in the classification of fixed point homogeneous manifolds are quotients by cyclic groups $\mathbb{Z}_q$, or finite subgroups of $SU(2)$ or $Sp(1)$ (see \cite{Searlediff}).  However, any finite group $\Gamma$ which acts freely on some $S^k$ ($k$ odd) may appear in the foliation setting. To see this, let $\FF$ be the SRF given by the join $\left(S^{2k+1}, \FF\right)\cong \left(S^k,\{pts\}\right)*\left(S^k,S^k\right)$. This is clearly point leaf maximal. Now the diagonal action of $\Gamma$ on $S^{2k+1}$ is free and preserves $\FF$, and hence $\FF$ descends to a point leaf maximal foliation on the space form $S^{2k+1}/\Gamma$. 
\end{rmk}\ \\

\bibliographystyle{abbrv}
\bibliography{References}

\end{document}